\documentclass[a4paper,11pt,leqno]{article}

\usepackage[english]{babel}
\usepackage{amsmath,amsthm,amssymb,thmtools,enumitem,mathrsfs,mathtools,hyperref,url,doi,titlesec}
\usepackage[left=2cm,right=2cm,top=2.5cm,bottom=2.5cm,bindingoffset=0cm]{geometry}

\setlist[enumerate]{topsep=0pt,label=\textup{(\roman*)},leftmargin=\parindent,labelsep=.5em}
\setlist{noitemsep}

\usepackage{quoting}
\quotingsetup{vskip=0pt}

\usepackage[T1]{fontenc}

\usepackage{tikz-cd}

\declaretheoremstyle[
  spaceabove=\topsep, spacebelow=6pt,
  headfont=\normalfont\bfseries,
  notefont=\mdseries, notebraces={(}{)},
  bodyfont=\normalfont\itshape,
  postheadspace=.5em,
  headformat=\NUMBER.~\NAME\NOTE,%
  qed=\qedsymbol
]{mystyle}

\declaretheoremstyle[
  spaceabove=\topsep, spacebelow=6pt,
  headfont=\normalfont\bfseries,
  notefont=\mdseries, notebraces={(}{)},
bodyfont=\normalfont,
  postheadspace=.5em,
  headformat=\NUMBER.~\NAME\NOTE,%
  qed=\qedsymbol
]{mydefstyle}

\theoremstyle{mystyle}
\declaretheorem[numberlike=subsection]{proposition}
\declaretheorem[numberlike=subsection]{theorem}
\declaretheorem[numberlike=subsection]{corollary}
\declaretheorem[numberlike=subsection]{lemma}

\theoremstyle{mydefstyle}

\declaretheorem[numberlike=subsection]{remark}

\numberwithin{equation}{subsection}

\titleformat{\section}[block]
  {\filcenter\normalfont\large\bfseries}{\thesection.}{1em}{}
\titleformat{\subsection}[runin]
  {\normalfont\itshape}{\textup{\bfseries\thesubsection.}}{.5em}{}
\titleformat{\subsubsection}[runin]
  {\normalfont\itshape}{\textup{\bfseries\thesubsubsection.}}{.5em}{}

\titlespacing*{\section} {0pt}{6ex plus 1ex minus .2ex}{3 ex plus .2ex}
\titlespacing*{\subsection} {0pt}{\topsep}{.5em}
\titlespacing*{\subsubsection} {0pt}{\topsep}{.5em}

\input{macs.sty}

\begin{document}

\begin{center}
\textbf{\Large A note on images of Galois representations}
\bigskip

{\large (with an application to a result of Litt)}
\bigskip

\textit{by}
\bigskip

{\Large Anna Cadoret and Ben Moonen}
\end{center}
\vspace{1cm}

{\small 

\noindent
\begin{quoting}
\textbf{Abstract.} Let $X$ be a variety (possibly non-complete or singular) over a finitely generated field~$k$ of characteristic~$0$. For a prime number~$\ell$, let $\rho_\ell$ be the Galois representation on the first $\ell$-adic cohomology of~$X$. We show that if $\ell$ varies the image of~$\rho_\ell$ is of bounded index in the group of $\bbZ_\ell$-points of its Zariski closure. We use this to improve a recent result of Litt about arithmetic representations of geometric fundamental groups. Litt's result says that there exist constants $N = N(X,\ell)$ such that every arithmetic representation $\pi_1(X_\kbar) \to \GL_n(\bbZ_\ell)$ that is trivial modulo~$\ell^N$ is unipotent. We show that these constants can in fact be chosen independently of~$\ell$.
\medskip

\end{quoting}

} 
\vspace{.5cm}

\section*{Introduction}

\noindent
Let $k$ be a field that is finitely generated over~$\bbQ$ and let $X$ be a geometrically connected variety over~$k$. We do not assume that $X$ is complete or non-singular. For $\ell$ a prime number, let $H_\ell = \mathrm{H}^1(X_\kbar,\bbZ_\ell)$ be the $\ell$-adic cohomology in degree~$1$, on which we have a Galois representation $\rho_\ell = \rho_{\ell,X} \colon \Gal(\kbar/k) \to \GL(H_\ell)$. The Zariski closure $\sG_\ell = \sG_{\ell,X}$ of $\Image(\rho_\ell)$ is an algebraic subgroup of~$\sGL(H_\ell)$. Though our understanding of the groups~$\sG_\ell$ is still incomplete, we dispose of several highly non-trivial results about their structure; see for instance \cite{SerreRibet}, \cite{LaPi1992} and~\cite{LaPi1995}.

In this paper we are mostly concerned with the actual images $\Image(\rho_\ell)$, and especially the way they vary with~$\ell$. It can be shown that $\Image(\rho_\ell)$ is open---and hence of finite index---in~$\sG_\ell(\bbZ_\ell)$ for every~$\ell$. Our first main result is that that the index $\bigl[\sG_\ell(\bbZ_\ell) : \Image(\rho_\ell)\bigr]$ is in fact bounded:
\vspace{\topsep}

\noindent
\textbf{Theorem A.} \emph{With $X/k$ as above, $\Image(\rho_\ell)$ is an open subgroup of\/~$\sG_\ell(\bbZ_\ell)$ for every~$\ell$ and the index $\bigl[\sG_\ell(\bbZ_\ell) : \Image(\rho_\ell)\bigr]$ is bounded when $\ell$ varies. The same is true if we everywhere replace $\mathrm{H}^1(X_\kbar,\bbZ_\ell)$ by $\mathrm{H}^1_{\text{c}}\bigl(X_\kbar,\bbZ_\ell(1)\bigr)$.}
\vspace{\baselineskip}

\noindent
In either variant ($\mathrm{H}^1$ or $\mathrm{H}^1_{\text{c}}(1)$) this result (Corollary~\ref{cor:indexbdd} in the text) is in fact a consequence of a similar, but more general, result about Galois representations associated with $1$-motives; see Theorem~\ref{thm:Main}. The proof of the result about $1$-motives occupies most of Section~\ref{sec:Indices}. It has three main ingredients. To handle the case of an abelian variety, in which case $\sG_\ell$ is reductive over~$\bbZ_\ell$ for almost all~$\ell$, we use a result of Wintenberger~\cite{WintenbLang} to control the derived group, and we use techniques from our previous paper~\cite{CadMoon} to control the abelian part. To perform the step from abelian varieties to general $1$-motives we then have to study the unipotent radicals of the groups~$\sG_\ell$; here we make essential use of results of Jossen~\cite{Jossen}.
\vspace{\baselineskip}

\noindent
We apply Theorem~A to improve a recent result of Litt~\cite{Litt}. Let $X/k$ be as above and assume $X$ is normal and geometrically integral. Following Litt, we say that a representation of the geometric fundamental group $\tau \colon \pi(X_\kbar) \to \GL_n(\bbZ_\ell)$ is \emph{arithmetic} if, possibly after replacing~$k$ by a finite extension, it appears as a subquotient of a representation of the arithmetic fundamental group~$\pi_1(X)$. (We leave out base points from the notation.) Litt's remarkable result is that we can obtain important information about arithmetic representations from their truncations modulo a power of~$\ell$ that does not depend on the representation. In particular he proves (\cite{Litt}, Theorem~1.2) that there exists a constant~$N$, depending on $X$ and~$\ell$ but not on~$\tau$, such that every arithmetic representation $\tau \colon \pi_1(X_\kbar) \to \GL_n(\bbZ_\ell)$ that is trivial modulo~$\ell^N$ is in fact unipotent. Our second main result gives the improvement that the dependence on~$\ell$ can be eliminated:
\vspace{\topsep}

\noindent
\textbf{Theorem B.} \emph{With $X/k$ as above, there exists an integer~$\ell_X$ such that for every prime $\ell \geq \ell_X$, every  arithmetic representation $\tau \colon \pi_1(X_\kbar) \to \GL_n(\bbZ_\ell)$ that is trivial modulo~$\ell$ is unipotent. In particular, in Litt's result the constant~$N$ can be chosen depending only on~$X$, independently of~$\ell$.}
\vspace{\baselineskip}

\noindent
The proof is given in Section~\ref{sec:Litt}.

\subsection*{Notation and conventions.}\label{ssec:Notat}
If $\sG$ is an algebraic group, $\sG^\der$ denotes its derived subgroup. We write~$\sG^\ab$ for $\sG/\sG^\der$ and $\ab \colon \sG \to \sG^\ab$ for the canonical map.

If $\sG$ is a group scheme over a field, $\sG^\idcomp$ denotes its identity component. If $\sG$ is a group scheme over a Dedekind domain with generic point $\eta$, then by $\sG^\idcomp$ we mean the closed subgroup scheme of~$\sG$ whose generic fibre is~$(\sG_\eta)^\idcomp$.

Unless indicated otherwise, reductive group schemes are assumed to have connected fibres. If $\sG$ is a reductive group scheme over a ring~$R$ (for us usually $R = \bbZ_\ell$), let $p \colon \sG^\sconn \to \sG^\der$ be the simply connected cover of its derived subgroup; then we define
\[
\sG^\der(R)_\mathrm{u} = \Image\bigl(p\colon \sG^\sconn(R) \to \sG^\der(R)\bigr)\, .
\]
By the rank of a reductive group over a connected base scheme we mean the absolute rank of its fibres.

If $R$ is a ring and $H$ is a free $R$-module of finite type, we denote by $\sGL(H)$ the associated reductive group over~$R$ and by $\GL(H) = \sGL(H)\bigl(R\bigr)$ the (abstract) group of $R$-linear automorphisms of~$H$.

\section{Indices of images of Galois representations in their Zariski closure}
Let $k$ be a finitely generated extension of $\bbQ$.
\label{sec:Indices}

\subsection{Notation.}
\label{ssec:1MotNot}
We refer to \cite{DelHodge3}, Section~10, and \cite{Jossen}, Section~1, for the basic notions of $1$-motives. Let $M$ be a $1$-motive over $k$. For $\ell$ a prime number, let $T_\ell(M)$ be the $\ell$-adic realization of~$M_\kbar$, and let
\[
\rho_{\ell,M} \colon \Gal(\kbar/k) \to \GL\bigl(T_\ell(M)\bigr)
\]
be the associated Galois representation. We denote by $\sG_{\ell,M} \subset \sGL\bigl(T_\ell(M)\bigr)$ the Zariski closure of the image of~$\rho_{\ell,M}$, which is a subgroup scheme of~$\sGL\bigl(T_\ell(M)\bigr)$, flat over~$\bbZ_\ell$.

\begin{theorem}\label{thm:Main}
With $M/k$ as above, $\Image(\rho_{\ell,M})$ is an open subgroup of ~$\sG_{\ell,M}(\bbZ_\ell)$ for every~$\ell$ and $\bigl[\sG_{\ell,M}(\bbZ_\ell) : \Image(\rho_{\ell,M})\bigr]$ is bounded when $\ell$ varies.
\end{theorem}

The proof consists of several steps. We start with a lemma.

\begin{lemma}\label{lem:GroupLemma}
\begin{enumerate}
\item\label{itm:Isogeny} Let $\sG$ be a reductive group over~$\bbZ_\ell$. Let $\sZ$ be the centre of\/~$\sG$, so that $q\colon \sZ^\idcomp \to \sG^\ab$ is an isogeny. Then there is a constant $C_1$, depending only  on the rank of\/~$\sG$, such that $q$ has degree at most~$C_1$.
\item\label{itm:scCover} Given a positive integer~$r$ there exist constants $C_2$ and $\ell_0$, depending only on~$r$, such that the following holds: For every prime number $\ell \geq \ell_0$ and every reductive group~$\sG$ over~$\bbZ_\ell$ of rank at most~$r$, the subgroup $\sG^\der(\bbZ_\ell)_\mathrm{u} \subset \sG^\der(\bbZ_\ell)$ (see the end of the Introduction for notation) has index at most~$C_2$.
\end{enumerate}
\end{lemma}

\begin{proof}
Let $r$ be given. It follows from the classification of split semisimple groups in terms of root data that there exists a constant $C = C(r) > 0$ such that for every reductive group~$\sG$ of rank $\leq r$ over~$\bbZ_\ell$ the isogeny $p \colon \sG^\sconn \to \sG^\der$ has degree at most~$C$ and also the centre of~$\sG^\der$, which is a finite flat $\bbZ_\ell$-group scheme, has order at most~$C$. Part~\ref{itm:Isogeny} follows because the kernel of~$q$ is contained in the centre of~$\sG^\der$.

For~\ref{itm:scCover}, take $\ell_0 = C(r)!$. Let $\ell$ be a prime number with $\ell > \ell_0$, and let $\sG$ be a reductive group over~$\bbZ_\ell$ of rank at most~$r$. Define $\mu_\sG = \Ker(p\colon \sG^\sconn \to \sG^\der)$. Then $\ell$ does not divide the order of the finite group scheme~$\mu_\sG$, which is therefore finite \'etale over~$\bbZ_\ell$. Writing $\bbF = \overline{\bbF}_\ell$, Galois cohomology gives a short exact sequence
\[
\sG^\sconn(\bbF) \tto \sG^\der(\bbF) \tto \mathrm{H}^1\bigl(\Gal(\bbF/\bbF_\ell),\mu_\sG(\bbF)\bigr) \to 1\, .
\]
Since $\Gal(\bbF/\bbF_\ell) \cong \Zhat$, the $\mathrm{H}^1$ that appears is a subquotient of $\mu_\sG(\bbF)$ (see \cite{SerreCL}, Section~XIII.1), and therefore $\bigl[\sG^\der(\bbF_\ell):\sG^\der(\bbF_\ell)_\mathrm{u}\bigr] \leq C(r)$. By Proposition~1 of~\cite{WintenbLang} it follows that $\bigl[\sG^\der(\bbZ_\ell): \sG^\der(\bbZ_\ell)_\mathrm{u}\bigr] \leq C(r)$.
\end{proof}

\subsection{The case of an abelian variety.}\label{AV} We now prove Theorem~\ref{thm:Main} when $M=A$ is an abelian variety.

\subsubsection{}
As a first step we reduce the problem to the case where the base field is a number field. Since $k$ is finitely generated over $\bbQ$, there exists an integral scheme~$S$ of finite type over~$\bbQ$, with generic point~$\eta$, and an abelian scheme $X \to S$ such that $k$ is the function field of~$S$ and $A$ is isomorphic to the generic fiber~$X_\eta$ of~$X/S$. Choose a prime number~$\ell$. By a result of Serre (see~\cite{SerreRibet}) there exists a closed point $s \in S$ such that the image of the Galois representation~$\rho_{\ell,X_s}$ is the same, via a specialization isomorphism $\spcl\colon T_{\ell}(X_\eta) \isomarrow T_\ell(X_s)$, as the image of~$\rho_{\ell,X_\eta}$. (Note that Serre's result can also be applied to finitely many prime numbers~$\ell$ at the same time, but not to infinitely many~$\ell$.) It then follows from \cite{Cadoret}, Theorem~1.2, that the image of the adelic Galois representation
\[
\rho_{\Zhat,X_s} \colon \Gal(\kbar/k) \to \prod_{\ell} \GL\bigl(T_\ell(X_s)\bigr)
\]
is open in the image of~$\rho_{\Zhat,X_\eta}$. This implies that $\sG_{\ell,X_s}(\bbZ_\ell)$ has bounded index in $\sG_{\ell,X_\eta}(\bbZ_\ell)$, where again we compare the two via the map~$\spcl$. So it suffices to prove the result for~$X_s$ over the number field~$\kappa(s)$.

\subsubsection{}
Now assume that $k$ is a number field. To simplify notation, write $\rho_\ell = \rho_{\ell,A}$ and $\sG_\ell = \sG_{\ell,A}$. As we may replace $k$ by a finite extension, we may further assume that the group schemes~$\sG_\ell$ have connected fibers. (See \cite{SerreRibet} or \cite{LaPi1992}, Proposition~6.14. It actually suffices to assume that all $n$-torsion points of~$A$ are $k$-rational for some $n \geq 3$.) Moreover, since by \cite{Bogomolov}, Theorem~1, we know that $\Image(\rho_\ell)$ is open in~$\sG_\ell(\bbZ_\ell)$ for every~$\ell$, we may exclude finitely many prime numbers~$\ell$ from our considerations. By \cite{LaPi1995}, Proposition~1.3 together with \cite{WintenbLang}, Theorem~2, the set~$\cL$ of primes numbers~$\ell$ for which the group scheme~$\sG_\ell$ is reductive and $\Image(\rho_\ell)$ contains $\sG_\ell^\der(\bbZ_\ell)_\unip$ contains all but finitely many~$\ell$. Hence it suffices to find a constant~$C$ such that $\bigl[\sG_\ell(\bbZ_\ell) : \Image(\rho_\ell)\bigr] \leq C$ for almost all $\ell \in \cL$.

Writing $\rho_\ell^\ab \colon \Gal(\kbar/k) \to \sG_\ell^\ab(\bbZ_\ell)$ for the composition of~$\rho_\ell$ and the canonical map $\ab\colon \sG_\ell (\bbZ_\ell)\to \sG_\ell^\ab(\bbZ_\ell)$, we have a commutative diagram with exact rows
\[
\begin{tikzcd}
1 \ar[r] & \Image(\rho_\ell) \cap \sG_\ell^\der(\bbZ_\ell) \ar[r] \ar[d,hook] & \Image(\rho_\ell) \ar[r] \ar[d,hook] & \Image(\rho_\ell^\ab) \ar[r] \ar[d,hook] & 1\\
1 \ar[r] & \sG_\ell^\der(\bbZ_\ell) \ar[r] & \sG_\ell(\bbZ_\ell) \ar[r] & \sG_\ell^\ab(\bbZ_\ell) &
\end{tikzcd}
\]
As the reductive groups~$\sG_\ell$, for $\ell \in \cL$, all have rank at most $2\cdot \dim(A)$, it follows from Lemma~\ref{lem:GroupLemma}\ref{itm:scCover} that there is an $\ell_0$ and a constant~$C_2$ such that $\Image(\rho_\ell) \cap \sG_\ell^\der(\bbZ_\ell)$ has index at most~$C_2$ in~$\sG_\ell^\der(\bbZ_\ell)$ for all $\ell \geq \ell_0$ in~$\cL$. It now only remains to find a bound for the index of~$\Image(\rho_\ell^\ab)$ in~$\sG_\ell^\ab(\bbZ_\ell)$.

\subsubsection{}
Choose an embedding $k \hookrightarrow \bbC$ and let $T_\Betti = \mathrm{H}_1(A_\bbC,\bbZ)$. Let $\sG_\Betti \subset \sGL(T_\Betti)$ be the (integral) Mumford--Tate group of~$A_\bbC$, which is part of a Shimura datum $(\sG_\Betti,X)$. Let $(\sG^\ab,X^\ab)$ be the associated abelian Shimura datum. Choose neat compact open subgroups $K \subset \sG_\Betti(\Zhat)$ and $K^\ab \subset \sG_\Betti^\ab(\Zhat)$ with $\ab(K) \subseteq K^\ab$. Possibly after replacing~$k$ with a finite extension (which we may do), there exists a level~$K$ structure on~$A$, and the choice of such allows us to associate with~$A$ a $k$-rational point~$s$ on the Shimura variety $\Sh_K(\sG_\Betti,X)$. Let $S$ denote the irreducible component of~$\Sh_K(\sG_\Betti,X)$ containing~$s$, and let $S^\ab \subset \Sh_{K^\ab}(\sG_\Betti^\ab,X^\ab)$ be its image.

In Section~3 of~\cite{CadMoon} we have defined representations $\phi \colon \pi_1(S) \to K$ and $\phi^\ab \colon \pi_1(S^\ab) \to K^\ab$ that fit into a commutative diagram
\[
\begin{tikzcd}
\pi_1(S) \ar[r,"\phi"] \ar[d] & K \ar[d,"\ab"]\\
\pi_1(S^\ab) \ar[r,"\phi^\ab"] & K^\ab
\end{tikzcd}
\]
(As explained in loc.\ cit., these representations are essentially independent of the choice of geometric base point, which we therefore omit from the notation.) Moreover, it is shown in ibid., Section~5, that the image of~$\phi$ in $K \subset \sG_\Betti(\Zhat)$ equals the image of~$\rho_{\Zhat,A}$ in $\prod_\ell\, \sG_\ell(\bbZ_\ell) \subset \sG_\Betti(\Zhat)$. By ibid., Proposition~3.5 and Corollary~3.7 (applied to $(\sG^\ab,X^\ab)$) it follows that the image of the composite homomorphism
\[
\Gal(\kbar/k) \xrightarrow{~\rho_\ell^\ab~} \sG_\ell^\ab(\bbZ_\ell) \tto \sG_\Betti^\ab(\bbZ_\ell)
\]
has bounded index in~$\sG_\Betti^\ab(\bbZ_\ell)$, for varying~$\ell$.

By \cite{UllYafQuebec}, Corollary~2.11, or \cite{Vasiu}, Theorem~1.3.1, the Mumford--Tate conjecture is true on connected centers. More precisely: under the inclusion $\sG_\ell \hookrightarrow \sG_\Betti \otimes \bbZ_\ell$ the connected center $\sZ_\ell^\idcomp$ of~$\sG_\ell$ maps into $\sZ_\Betti \otimes \bbZ_\ell$ and induces an isomorphism $\sZ_\ell^\idcomp \isomarrow \sZ_\Betti^\idcomp \otimes \bbZ_\ell$. (We only need to consider the primes $\ell \in \cL$, for which $\sZ_\ell^\idcomp$ and~$\sZ_\Betti^\idcomp \otimes \bbZ_\ell$ are tori over~$\bbZ_\ell$.) By Lemma~\ref{lem:GroupLemma}\ref{itm:Isogeny} it follows that the natural homomorphisms $\sG_\ell^\ab \to \sG_\Betti^\ab \otimes \bbZ_\ell$ are isogenies of bounded degree, and therefore the index of~$\Image(\rho_\ell^\ab)$ in~$\sG_\ell^\ab(\bbZ_\ell)$ is bounded when $\ell$ varies. This completes the proof of Theorem~\ref{thm:Main} in the case where $M = A$ is an abelian variety.

\begin{remark}
The proof shows that for $\ell \gg 0$ (depending on~$A$) the index $\bigl[\sG_\ell(\bbZ_\ell) : \Image(\rho_\ell)\bigr]$ can be bounded by a constant that only depends on the Mumford--Tate group~$\sG_\Betti$ and the abelianized Shimura datum $(\sG_\Betti^\ab,X^\ab)$.
\end{remark}

\subsection{The case of a split $1$-motive.}
Next, we consider a split $1$-motive $M = [Y \xrightarrow{0} A \times T]$ (i.e., a $1$-motive in which all extensions are trivial).

As before we may replace~$k$ with a finite extension. We may therefore assume that $T$ is a split torus and that $Y$ is constant. Then the projection $\sG_{\ell,M} \to \sG_{\ell,A\times T}$ is an isomorphism and restricts to an isomorphism $\Image(\rho_{\ell,M}) \isomarrow \Image(\rho_{\ell,A\times T})$; this reduces the problem to the case $Y=0$.

The case where $T$ is trivial is treated in \ref{AV}, so we assume $T \neq \{1\}$. As $T$ is a split torus, the Galois group $\Gal(\kbar/k)$ acts on $T_\ell(T)$ through the $\ell$-adic cyclotomic character~$\chi_\ell$. If the abelian variety~$A$ is zero, then $\Image(\rho_{\ell,M}) \isomarrow \Image(\chi_\ell)$ for all~$\ell$, and the fact that $k\cap \bbQ^\ab$ is a finite extension of~$\bbQ$ gives the desired conclusion that $\Image(\rho_{\ell,M})$ is of bounded index in~$\bbZ_\ell^\times$.

Assume then that $Y=0$ and $A \neq 0$. The group scheme $\sG_{\ell,M}$ is a closed subgroup scheme of $\sG_{\ell,A} \times \sG_{\ell,T}$. Let $V_\ell(A) = T_\ell(A) \otimes \bbQ_\ell$. Choose a polarization of~$A$, and let $\psi_\ell \colon V_\ell(A) \times V_\ell(A) \to \bbQ_\ell(1)$ be the associated alternating bilinear form. The image of~$\rho_{\ell,A}$ is contained in the group of symplectic similitudes $\CSp\bigl(V_\ell(A),\psi_\ell\bigr)$, and if $\nu \colon \CSp\bigl(V_\ell(A),\psi_\ell\bigr) \to \bbQ_\ell^\times$ is the multiplier character, we have the relation $\nu \circ \rho_{\ell,A} = \chi_\ell$. It follows that the image of~$\rho_{\ell,M}$ is contained in the graph of~$\nu$, viewed as a subgroup of $\sG_{\ell,A}(\bbQ_\ell) \times \sG_{\ell,T}(\bbQ_\ell)$ and hence the projection map $\sG_{\ell,M}(\bbQ_\ell) \to \sG_{\ell,A}(\bbQ_\ell)$ is injective. This gives us a commutative diagram
\[
\begin{tikzcd}
\Image(\rho_{\ell,M}) \ar[r,hook] \ar[d,"\wr"] & \sG_{\ell,M}(\bbZ_\ell) \ar[r,hook] \ar[d,hook] & \sG_{\ell,M}(\bbQ_\ell) \ar[d,hook]\\
\Image(\rho_{\ell,A}) \ar[r,hook] & \sG_{\ell,A}(\bbZ_\ell) \ar[r,hook] & \sG_{\ell,A}(\bbQ_\ell)
\end{tikzcd}
\]
and it follows that $\bigl[\sG_{\ell,M}(\bbZ_\ell) : \Image(\rho_{\ell,M})\bigr] \leq \bigl[\sG_{\ell,A}(\bbZ_\ell) : \Image(\rho_{\ell,A})\bigr]$. The theorem for~$M$ now follows from the case of an abelian variety treated in \ref{AV}.

\subsection{The general case.}
As a last step in the proof, we now turn to the case of a general $1$-motive $M = [Y \to G]$.

\subsubsection{}
The semi-abelian variety~$G$ is an extension of an abelian variety~$A$ by a torus~$T$. On~$M$ we have a weight filtration whose graded pieces are $T$, $A$ and~$Y$, respectively. Let $\tilde{M} = T \oplus A \oplus Y$ (or, in different notation, $\tilde{M} = [Y \xrightarrow{0} A\times T]$) be the associated total graded, which is a split $1$-motive.

As before we may replace~$k$ with a finite extension. We therefore can, and from now on will, assume that $T$ is a split torus, that $Y$ is constant, and that all $3$-torsion points of~$A$ are $k$-rational. This implies that the group schemes $\sG_{\ell,M}$ and~$\sG_{\ell,\tilde{M}}$ have connected fibers.

\subsubsection{}
Let $W$ denote the weight filtration on~$T_\ell(M)$. The Galois representation~$\rho_{\ell,M}$ takes values in the parabolic subgroup $\sStab_W \subset \sGL\bigl(T_\ell(M)\bigr)$. We have a natural identification $\gr^W\bigl(T_\ell(M)\bigr) \isomarrow T_\ell(\tilde{M})$ and if $\pi \colon \sStab_W \to \sGL\bigl(T_\ell(\tilde{M})\bigr)$ is the natural homomorphism then we have $\pi \circ \rho_{\ell,M} = \rho_{\ell,\tilde{M}}$. In particular, $\pi$~restricts to a homomorphism $\sG_{\ell,M} \to \sG_{\ell,\tilde{M}}$. Let $\sU_{\ell,M}$ denote the kernel of the latter homomorphism, which is the intersection of~$\sG_{\ell,M}$ with the unipotent radical of~$\sStab_W$.

Let $k\subset k^\unip$ be the field extension (inside~$\kbar$) that corresponds with the kernel of~$\rho_{\ell,\tilde{M}}$. Then $\rho_{\ell,M}$ restricts to a Galois representation $\rho_{\ell,M}^\unip \colon \Gal(\kbar/k^\unip) \to \sU_{\ell,M}(\bbZ_\ell)$. This gives us a commutative diagram with exact rows
\[
\begin{tikzcd}
1 \ar[r] & \Image(\rho_{\ell,M}^\unip) \ar[r] \ar[d,hook] & \Image(\rho_{\ell,M}) \ar[r] \ar[d,hook] & \Image(\rho_{\ell,\tilde{M}}) \ar[r] \ar[d,hook] & 1\\
1 \ar[r] & \sU_{\ell,M}(\bbZ_\ell) \ar[r] & \sG_{\ell,M}(\bbZ_\ell) \ar[r] & \sG_{\ell,\tilde{M}}(\bbZ_\ell) &
\end{tikzcd}
\]
which gives the inequality
\[
\bigl[\sG_{\ell,M}(\bbZ_\ell) : \Image(\rho_{\ell,M}) \bigr] \leq  \bigl[\sU_{\ell,M}(\bbZ_\ell) : \Image(\rho_{\ell,M}^\unip)\bigr] \cdot \bigl[\sG_{\ell,\tilde{M}}(\bbZ_\ell) : \Image(\rho_{\ell,\tilde{M}}) \bigr]\, .
\]
(At this point we do not yet know that these numbers are all finite.)

\subsubsection{}
\label{ssec:LastStep}
By the previous part of the proof, $\bigl[\sG_{\ell,\tilde{M}}(\bbZ_\ell) : \Image(\rho_{\ell,\tilde{M}}) \bigr]$ is bounded independently of~$\ell$. What remains to be shown is that also the index of $\Image(\rho_{\ell,M}^\unip)$ in $\sU_{\ell,M}(\bbZ_\ell)$ is finite and bounded when $\ell$ varies. Our proof of this is based on the results of Jossen~\cite{Jossen}. Choose a field embedding $k \to \bbC$. Let $T_\Betti(M)$ denote the Hodge realization of~$M_\bbC$, and let $\sG_{\Betti,M} \subset \sGL\bigl(T_\Betti(M)\bigr)$ denote the integral Mumford--Tate group. Again we have a weight filtration~$W$ on~$T_\Betti(M)$, and an associated parabolic subgroup $\sStab_W \subset \sGL\bigl(T_\Betti(M)\bigr)$. Analogous to how we defined~$\sU_{\ell,M}$, let $\sU_{\Betti,M}\subset \sG_{\Betti,M}$ be the intersection of $ \sG_{\Betti,M}$ with the unipotent radical of  $\sStab_W $.

It follows from the results of Brylinski in~\cite{Bryl} (which extend results of Deligne for abelian varieties), that under the comparison isomorphism $T_\Betti(M) \otimes \bbZ_\ell \isomarrow T_\ell(M)$ we have $\sG_{\ell,M} \subseteq \sG_{\Betti,M} \otimes \bbZ_\ell$. (See also \cite{Jossen}, Theorem~3.1.) Hence $\sU_{\ell,M} \subseteq \sU_{\Betti,M} \otimes \bbZ_\ell$. Writing $\mathfrak{u}_?$ for the Lie algebra of~$\sU_?$, this gives an inclusion of $\bbZ_\ell$-Lie algebras $\mathfrak{u}_{\ell,M} \subseteq \mathfrak{u}_{\Betti,M} \otimes \bbZ_\ell$.

Possibly after again replacing $k$ with a finite extension, we may assume that all transformations $\phi \in \GL\bigl(T_2(M)\bigr)$ that lie in the image of the $2$-adic unipotent representation~$\rho_{2,M}^\unip$ have the property that $(\phi - 1)\colon T_2(M) \to T_2(M)$ is divisible by~$2$. For every~$\ell$ we then have a map
\[
\vartheta_\ell = \log \circ \rho_{\ell,M}^\unip \colon \Gal(\kbar/k^\unip) \to \mathfrak{u}_{\ell,M}\, ,
\]
whose image is a $\bbZ_\ell$-submodule of~$\mathfrak{u}_{\ell,M}$.

Write $V_\Betti(?) = T_\Betti(?) \otimes \bbQ$ and $V_\ell(?) = T_\ell(?) \otimes \bbQ_\ell$. Let $P = P(M)$ be the semi-abelian variety of \cite{Jossen}, Definition~4.3. By ibid., Theorem~6.2, we have an isomorphism $\alpha_\Betti\colon V_\Betti(P) \isomarrow (\mathfrak{u}_{\Betti,M} \otimes \bbQ)$. Moreover, by the first assertion of ibid., Theorem~7.2, this extends to a commutative diagram
\[
\begin{tikzcd}
V_\ell(P) \ar[r,"\sim"] \ar[d,"\wr"] & \Image(\vartheta_\ell) \otimes \bbQ_\ell \ar[r,hook] & \mathfrak{u}_{\ell,M} \otimes_{\bbZ_\ell} \bbQ_\ell \ar[d,hook] \\
V_\Betti(P) \otimes \bbQ_\ell \ar[rr,"\sim","\alpha_\Betti \otimes 1"'] && \mathfrak{u}_{\Betti,M} \otimes_\bbZ \bbQ_\ell
\end{tikzcd}
\]
and it follows that the inclusion maps $\Image(\vartheta_\ell) \otimes \bbQ_\ell \hookrightarrow  \mathfrak{u}_{\ell,M} \otimes_{\bbZ_\ell} \bbQ_\ell \hookrightarrow \mathfrak{u}_{\Betti,M} \otimes_\bbZ \bbQ_\ell$ are isomorphisms. This already implies that $\Image(\rho_{\ell,M}^\unip)$ has finite index in $\sU_{\ell,M}(\bbZ_\ell)$. To bound this index when $\ell$ varies, we use that by the last assertion of ibid., Theorem~7.2, for almost all~$\ell$ the previous diagram restricts to a diagram
\[
\begin{tikzcd}
T_\ell(P) \ar[r,"\sim"] \ar[d,"\wr"] & \Image(\vartheta_\ell) \ar[r,hook] & \mathfrak{u}_{\ell,M} \ar[d,hook] \\
T_\Betti(P) \otimes \bbZ_\ell \ar[d,hook] && \mathfrak{u}_{\Betti,M} \otimes \bbZ_\ell \ar[d,hook]\\
V_\Betti(P) \otimes \bbQ_\ell \ar[rr,"\sim","\alpha_\Betti \otimes 1"'] && \mathfrak{u}_{\Betti,M} \otimes_\bbZ \bbQ_\ell
\end{tikzcd}
\]
As the index of the lattice $T_\Betti(P)$ inside $T_\Betti(P) + \mathfrak{u}_{\Betti,M}$ (taken inside $\mathfrak{u}_{\Betti,M} \otimes \bbQ$) is finite and independent of~$\ell$, it follows that also the index of~$\Image(\vartheta_\ell)$ inside~$\mathfrak{u}_{\ell,M}$ is bounded when $\ell$ varies. Hence the index of $\Image(\rho_{\ell,M}^\unip)$ in $\sU_{\ell,M}(\bbZ_\ell)$ is finite and bounded when $\ell$ varies. The proof of Theorem~\ref{thm:Main} is now complete.

\begin{corollary}
\label{cor:indexbdd}
Let $X$ be a geometrically connected variety over a field~$k$ that is finitely generated over~$\bbQ$. For $\ell$ a prime number, write $H_\ell = \mathrm{H}^1(X_\kbar,\bbZ_\ell)$, let $\rho_{\ell,X} \colon \Gal(\kbar/k) \to \GL(H_\ell)$ be the natural Galois representation, and let\/ $\sG_{\ell,X} \subset \sGL(H_\ell)$ be the Zariski closure of\/~$\Image(\rho_{\ell,X})$. Then\/ $\Image(\rho_{\ell,X})$ is an open subgroup of\/~$\sG_{\ell,X}(\bbZ_\ell)$ for every~$\ell$ and $\bigl[\sG_{\ell,X}(\bbZ_\ell) : \Image(\rho_{\ell,X})\bigr]$ is bounded when $\ell$ varies. The same assertion is true if we everywhere take $H_\ell = \mathrm{H}^1_{\text{c}}\bigl(X_\kbar,\bbZ_\ell(1)\bigr)$.
\end{corollary}

\begin{proof}
We may pass to the dual Galois representation~$H_\ell^\vee$. But by \cite{BaVSri}, Corollary~5.3.2,\footnote{It appears that the definition of $\ell$-adic homology as given in \cite{BaVSri}, Section~2.5, needs to be changed: writing $\Lambda = \bbZ/m\bbZ$ one should define $\mathrm{H}_i(X_\kbar,\Lambda)$ to be $\cH_i\bigl(\mathrm{RHom}(\mathrm{R}\Gamma(X_\kbar,\Lambda),\Lambda)\bigr)$, whereas loc.\ cit.\ uses $\mathrm{RHom}(~,\Lambda(-n)[-2n])$ with $n=\dim(X)$.} $H_\ell^\vee$ is the $\ell$-adic realization of the $1$-motive $\Alb^-(X)$ of $X/k$ as in ibid., Section~7.2; so Theorem~\ref{thm:Main} applies.

For the version with $\mathrm{H}^1_{\text{c}}\bigl(X_\kbar,\bbZ_\ell(1)\bigr)$ we apply Theorem~\ref{thm:Main} to the $1$-motive $\mathrm{H}^1_{\text{m}}(X)(1)$ introduced in \cite{DelHodge3}, Section~10.3.
\end{proof}

\begin{remark}
\label{rem:PicCurve}
In the next section we apply this result with $X$ a non-singular curve and $H_\ell = \mathrm{H}^1(X_\kbar,\bbZ_\ell)$. In that case we can be much more explicit about the $1$-motive whose $\ell$-adic realization gives the Galois representation~$H_\ell^\vee$. Namely, if $X \hookrightarrow \bar{X}$ is the complete non-singular model of~$X$, we can form the curve~$X^\prime$ that is obtained from~$\bar{X}$ by identifying all points in~$\bar{X}\setminus X$. Then $\Pic^0_{X^\prime/k}$ is a semi-abelian variety whose $\ell$-adic realization is isomorphic to~$H_\ell^\vee$.
\end{remark}

\section{Uniformity in~$\ell$ in a theorem of Litt}
\label{sec:Litt}

\subsection{}
\label{ssec:LittSetup}
As before, let $k$ be a field that is finitely generated over~$\bbQ$. Let $X$ be a normal scheme, geometrically integral, separated and of finite type over $k$, and let~$\bar{x}$ be a geometric point on~$X$. The main result of~\cite{Litt} is the following.

\begin{theorem}[Litt]
\label{thm:Litt}
Let $X/k$ be as in~\emph{\ref{ssec:LittSetup}}, and let $\ell$ be a prime number. Then there exists a positive integer~$N$, depending on $X$ and~$\ell$, such that for any arithmetic representation $\tau \colon \pi_1(X_{\bar{k}})\to \GL_n(\bbZ_\ell)$ we have
\[
\text{$\tau$ is trivial modulo $\ell^N$} \quad\implies\quad \text{$\tau$ is unipotent.}
\]
\end{theorem}

If~$\tau$ is a monodromy representation of a smooth proper family over~$X$, it is known that $\tau$ is semisimple; in this case, therefore, the conclusion is that if $\tau$ is trivial modulo~$\ell^N$ then $\tau$ is trivial. See \cite{Litt}, Corollary~1.6.

\subsection{}
\label{ssec:QlDef}
For $X/k$ as in~\ref{ssec:LittSetup}, let $N(X,\ell)$ be the smallest positive integer~$N$ such that all arithmetic representations $\tau \colon \pi_1(X_{\bar{k}}) \to \GL_n(\bbZ_\ell)$ that are trivial modulo~$\ell^N$ are unipotent. Our goal is to estimate~$N(X,\ell)$ and to show that it is bounded as a function of~$\ell$. Note that nothing changes if we replace~$\bar{x}$ by a different geometric base point or if we replace~$k$ by a finite extension. In what follows we may therefore assume that~$\bar{x}$ lies over a $k$-rational point $x \in X(k)$. This gives an action of $\Gal(\kbar/k)$ on~$\pi_1(X_{\bar{k}})$ and an isomorphism $\pi_1(X) \cong \pi_1(X_{\bar{k}}) \rtimes \Gal(\kbar/k)$.

By a Bertini argument (see \cite{DelPi1}, Lemma~1.4 and \cite{Litt}, Section~4.1) there exists an affine smooth curve~$C/k$ and a $k$-morphism $C\to X$ such that the induced homomorphism $\pi_1(C_\kbar) \to \pi_1(X_\kbar)$ is surjective. For such a curve we have $N(X,\ell) \leq N(C,\ell)$ for all~$\ell$. What we will estimate is~$N(C,\ell)$. In what follows we therefore assume that $X$ is an affine curve, smooth over~$k$. Moreover, if $X \hookrightarrow \bar{X}$ is the complete non-singular model of~$X$ we may assume that $\bar{X}$ has positive genus. (This is not essential but it will simplify some later assertions.)

\subsection{}
Choose an embedding $k \hookrightarrow \bbC$. As remarked in~\ref{rem:PicCurve} there is a semi-abelian variety $M = \Pic^0_{X^\prime/k}$ whose Hodge and $\ell$-adic realizations are given by $T_\Betti(M) = \mathrm{H}_1(X_\bbC,\bbZ)$ and $T_\ell(M) = \mathrm{H}_1(X_\kbar,\bbZ_\ell)$. As in Section~\ref{sec:Indices}, let $\sG_{\Betti,M} \subset \sGL\bigl(T_\Betti(M)\bigr)$ be the (integral) Mumford--Tate group and $\sG_{\ell,M} \subset \sGL\bigl(T_\ell(M)\bigr)$ the Zariski closure of the image of $\rho_{\ell,X} \colon \Gal(\kbar/k) \to \GL\bigl(T_\ell(M)\bigr)$. If there is no risk of confusion, we will from now on omit~$M$ from the notation. As already remarked, to estimate the constants~$N(X,\ell)$ we may replace~$k$ with a finite extension; we may therefore assume that the group schemes~$\sG_\ell$ have connected fibres.

For every~$\ell$ we have a comparison isomorphism $\iota_\ell \colon T_\Betti \otimes \bbZ_\ell \isomarrow T_\ell$, which we will take as an identification. As $M$ is semi-abelian, the weight filtration~$W_\bullet$ on~$T_\Betti$ only has two steps:
\[
W_{-3} = 0 \quad\subset\quad  W_{-2} = \Ker\bigl(\mathrm{H}_1(X_\bbC,\bbZ) \to \mathrm{H}_1(\bar{X}_\bbC,\bbZ)\bigr) \quad\subset\quad  W_{-1} = T_\Betti(M)\, .
\]
The weight filtration on~$T_\ell$ is $W_\bullet(T_\Betti) \otimes \bbZ_\ell$. The Mumford--Tate group~$\sG_\Betti$ is contained in the stabilizer $\sStab_W \subset \sGL(T_\Betti)$ of the weight filtration, and as we have already seen in~\ref{ssec:LastStep}, $\sG_\ell \hookrightarrow \sG_\Betti \otimes \bbZ_\ell$ for all~$\ell$.

Let $\sQ_\ell \subset \sG_\ell$ be the subgroup scheme of elements $g \in \sG_\ell$ for which there is a scalar~$\alpha$ such that $g$ acts on $\gr^W_{-i}$ as $\alpha^i \cdot \id$ ($i=1,2$). Let $\nu \colon \sQ_\ell \to \bbG_\mult$ be the character (over~$\bbZ_\ell$) given by $g \mapsto \alpha$.

For our proof of Theorem~\ref{thm:LittUniform} we need the following result.

\begin{lemma}
\label{lem:nusurj}
For almost all~$\ell$ the homomorphism $\nu \colon \sQ_\ell(\bbZ_\ell) \to \bbZ_\ell^\times$ is surjective.
\end{lemma}

\begin{proof}
Analogous to how we defined the group schemes~$\sQ_\ell$, let $\sQ_\Betti \subset \sG_\Betti$ be the subgroup scheme of elements $g \in \sG_\Betti$ that act on~$\gr^W_{-i}$ as $\alpha^i \cdot \id$ for some scalar~$\alpha$, and let $\nu \colon \sQ_\Betti \to \bbG_\mult$ (over~$\bbZ$) be given by $g \mapsto \alpha$.

The unipotent radical of $\sStab_W \subset \sGL(H_\Betti)$ is the vector group scheme associated with the free $\bbZ$-module $\Hom(\gr_W^{-1},\gr_W^{-2})$. The unipotent radical of $\sG_\Betti \otimes \bbQ$ is the intersection of $\sG_\Betti \otimes \bbQ$ and $\mathrm{R}_\unip(\sStab_W)$ and is therefore again a vector group scheme. We denote it by~$\sU_{\Betti,\bbQ}$. (It is the generic fibre of the $\sU_{\Betti,M}$ of~\ref{ssec:LastStep}.) We have an isomorphism $h_\bbQ \colon \sQ_\Betti \otimes \bbQ \isomarrow \sU_{\Betti,\bbQ} \rtimes \bbG_\mult$, with $z \in \bbG_\mult$ acting on~$\sU_{\Betti,\bbQ}$ as multiplication by~$z$. This extends over an open part of~$\Spec(\bbZ)$, say over $R = \bbZ[1/n]$; by this we mean that $\sG_{\Betti} \otimes R$ is smooth over~$R$, its unipotent radical~$\sU_{\Betti,R}$ is the vector group scheme given by a free $R$-submodule of $\Hom_R(\gr_W^{-1} \otimes R,\gr_W^{-2} \otimes R)$, and $h_\bbQ$ extends to an isomorphism $h_R \colon \sQ_\Betti \otimes R \isomarrow \sU_{\Betti,R} \rtimes \bbG_\mult$.

Similarly, the unipotent radical of $\sG_\ell \otimes \bbQ_\ell$ is a vector group scheme~$\sU_{\ell,\bbQ_\ell}$ that is associated with a $\bbQ_\ell$-subspace of $\Hom_{\bbQ_\ell}(\gr_W^{-1} \otimes \bbQ_\ell,\gr_W^{-2} \otimes \bbQ_\ell)$. We have $\sU_{\ell,\bbQ_\ell} \rtimes \{1\} \subseteq \sQ_\ell \otimes \bbQ_\ell \subseteq \sU_{\ell,\bbQ_\ell} \rtimes \bbG_\mult$, where in the semi-direct product $z \in \bbG_\mult$ again acts on~$\sU_{\ell,\bbQ_\ell}$ as multiplication by~$z$. By \cite{Jossen}, Theorems~6.2 and~7.2, the inclusion $\sG_\ell \hookrightarrow \sG_\Betti \otimes \bbZ_\ell$ restricts to an isomorphism $\sU_{\ell,\bbQ_\ell} \isomarrow \sU_{\Betti,\bbQ} \otimes \bbQ_\ell$.

We claim that the homomorphism $\nu \colon \sQ_\ell \otimes \bbQ_\ell \to \bbG_{\mult,\bbQ_\ell}$ is surjective (as a homomorphism of algebraic groups) for all~$\ell$. If this is true, it follows that the inclusions $\sG_\ell \hookrightarrow \sG_\Betti \otimes \bbZ_\ell$ restrict to isomorphisms $\sQ_\ell \otimes \bbQ_\ell \isomarrow \sQ_\Betti \otimes \bbQ_\ell$, and hence also to isomorphisms $\sQ_\ell \isomarrow \sQ_\Betti \otimes \bbZ_\ell$ for all $\ell \nmid n$. As $\sQ_\Betti \otimes \bbZ_\ell \cong \sU_{\Betti,\bbZ_\ell} \rtimes \bbG_\mult$, with $\nu$ given by the second projection, this gives the desired conclusion that $\nu \colon \sQ_\ell(\bbZ_\ell) \to \bbZ_\ell^\times$ is surjective for all $\ell \nmid n$.

It remains to prove the claim. There exists an integral affine scheme~$S$ of finite type over~$\bbZ$ whose function field is~$k$, such that $M$ extends to a semi-abelian variety over~$S$. Let $s \in S$ be a closed point with residue field of cardinality~$q(s)$ such that $\ell \nmid q(s)$, and let $F_s \in \Gal(\kbar/k)$ be an arithmetic Frobenius element at~$s$. Associated with~$F_s$ we have a Frobenius torus~$\sT(F_s)$ over~$\bbQ$ whose character group is isomorphic to the $\Gal(\Qbar/\bbQ)$-submodule of~$\smash{\Qbar}^\times$ generated by the eigenvalues $\alpha_1,\ldots,\alpha_m$ of~$F_s$ acting on $\gr^W_{-1}\bigl(T_\ell\bigr) \otimes \Qbar_\ell$. (See~\cite{SerreRibet}. Note that $F_s$ acts on~$\gr^W_{-2}$ as multiplication by $q(s) \in \bbZ_\ell^\times$, and since we have assumed that $g(\bar{X}) > 0$ we only have to consider the eigenvalues of~$F_s$ on~$\gr^W_{-1}$.) Because $F_s$ acts semi-simply on $T_\ell \otimes \bbQ_\ell$ (cf.\ \cite{Litt}, Lemma~2.9), we obtain an injective homomorphism $i\colon \sT(F_s) \otimes \bbQ_\ell \hookrightarrow \sG_\ell \otimes \bbQ_\ell$. Further, we have a homomorphism $j \colon \bbG_\mult \hookrightarrow \sT(F_s)$ corresponding to the $\Gal(\Qbar/\bbQ)$-equivariant map $X^*\bigl(\sT(F_s)\bigr) \to \bbZ$ that sends each~$\alpha_i$ to~$1$. Then the image of $(i\circ j) \colon \bbG_\mult \to \sG_\ell \otimes \bbQ_\ell$ is contained in~$\sQ_\ell \otimes \bbQ_\ell$ and $i \circ j$ gives a section of the map~$\nu$. So indeed $\nu \colon \sQ_\ell \otimes \bbQ_\ell \to \bbG_{\mult,\bbQ_\ell}$ is surjective. This completes the proof.
\end{proof}

\subsection{Notation.}
For $\ell$ a prime number and $\alpha \in \bbZ_\ell^\times$, define  $C(\alpha,\ell)$ by
\[
C(\alpha,\ell) = \begin{dcases}
\frac{1}{s} \cdot \left(v_\ell(\alpha^s-1) + \frac{1}{\ell-1} + 2\right) & \text{if $\ell = 2$} \\
\frac{1}{s} \cdot \left(v_\ell(\alpha^s-1) + \frac{1}{\ell-1}\right) & \text{if $\ell > 2$}
\end{dcases}
\]
where $v_\ell$ denotes the $\ell$-adic valuation and where $s$ is the order of $(\alpha \bmod 4) \in (\bbZ/4\bbZ)^\times$ if $\ell = 2$ and is the order of $(\alpha \bmod \ell) \in (\bbZ/\ell\bbZ)^\times$ if $\ell > 2$.

Note that for $\ell > 2$ we can choose a root of unity $\zeta \in \bbZ_\ell^\times$ of order $\ell-1$, and then $\alpha$ can be written as $\alpha = \zeta^{(\ell-1)/s} \cdot \exp(y)$ for some $y \in \ell\bbZ_\ell$. With this notation, $C(\alpha,\ell) = (1/s) \cdot \bigl(v_\ell(y) + 1/(\ell-1)\bigr)$.

\begin{proposition}
\label{prop:2CBound}
Let $\ell > 2$. Suppose $\Image(\rho_{\ell,X}) \cap \sQ_\ell(\bbZ_\ell)$ contains an element~$g$ such that $\alpha = \nu(g) \in \bbZ_\ell^\times$ has infinite order. Then $N(X,\ell) \leq 1 + \bigl\lfloor 2\cdot C(\alpha,\ell)\bigr\rfloor$.
\end{proposition}

\begin{proof}
All we need to do is to carefully go through the proof of Theorem~1.2 in~\cite{Litt}. For the reader's convenience, let us give the steps that are required to extract the assertion from Litt's paper.

Let $g$ and $\alpha = \nu(g)$ be as in the assertion. Let $N = 1 + \bigl\lfloor 2\cdot C(\alpha,\ell)\bigr\rfloor$. The claim is that every arithmetic representation $\tau \colon \pi_1(X_{\bar{k}}) \to \GL_n(\bbZ_\ell)$ that is trivial modulo~$\ell^N$ is unipotent. As explained by Litt in \cite{Litt}, Section~4.1 (especially Lemma~4.1 and the first half of the proof of his Theorem~1.2), it suffices to prove this only for those representations~$\tau$ that extend to a representation $  \pi_1(X) \to \GL_n(\bbZ_\ell)$.

Let $\pi_1(X_{\bar{k}})^{(\ell)}$ be the maximal pro-$\ell$ quotient of~$\pi_1(X_{\bar{k}})$, let $\bbZ_\ell\lbb \pi_1(X_{\bar{k}})^{(\ell)} \rbb$ be the completed group algebra, $\cI \subset \bbZ_\ell\lbb \pi_1(X_{\bar{k}})^{(\ell)} \rbb$ the augmentation ideal, and
\[
\bbQ_\ell\lbb \pi_1(X_{\bar{k}})^{(\ell)} \rbb = \lim_n \Bigl(\bbQ_\ell \otimes\bigl(\bbZ_\ell\lbb \pi_1(X_{\bar{k}})^{(\ell)} \rbb/\cI^n\bigr)\Bigr)\, .
\]
These rings come equipped with an action of~$\Gal(\kbar/k)$.

Let $r$ be a real number with $\bigl\lfloor 2\cdot C(\alpha,\ell)\bigr\rfloor < r < N$. Litt defines (\cite{Litt}, Definition~3.2) a Galois-stable $\bbQ_\ell$-subalgebra
\[
\bbQ_\ell\lbb \pi_1(X_{\bar{k}})^{(\ell)} \rbb^{\leq \ell^{-r}} \subset \bbQ_\ell\lbb \pi_1(X_{\bar{k}})^{(\ell)} \rbb\,
\]
which he calls the \emph{convergent group ring}. The representation~$\tau$ gives rise to a Galois-equivariant homomorphism $\beta\colon \bbZ_\ell\lbb \pi_1(X_{\bar{k}})^{(\ell)} \rbb \to M_n(\bbZ_\ell)$, where the Galois action on $M_n(\bbZ_\ell)$ is obtained by using the section of $\pi_1(X) \to \Gal(\kbar/k)$ associated with the rational base point $x \in X(k)$, together with the assumption that $\tau$ extends to a representation of~$\pi_1(X)$. Litt shows (ibid., Proposition~3.4) that the assumption that $\tau$ is trivial modulo~$\ell^N$ implies that $\beta$ uniquely extends to a $\Gal(\kbar/k)$-equivariant homomorphism
\[
\tilde\beta \colon \bbQ_\ell\lbb \pi_1(X_{\bar{k}})^{(\ell)} \rbb^{\leq \ell^{-r}} \to M_n(\bbQ_\ell)\, .
\]

The convergent group ring comes equipped with a weight filtration~$W_\bullet$ and a Gauss norm for which $\tilde\beta$ is continuous. Write $g = \rho_{\ell,X}(\sigma)$ for some $\sigma \in \Gal(\kbar/k)$, and recall that $\alpha = \nu(g)$. The proof of \cite{Litt}, Theorem~2.8 (at the end of Section~2) shows that $\sigma$ acts on~$\gr_{-i}^W$ as multiplication by~$\alpha^i$. It then follows from ibid., Theorem~3.6 and Remark~3.11, that if we consider the eigenspaces of~$\sigma$ acting on $W_{-n} = W_{-n} \bbQ_\ell\lbb \pi_1(X_{\bar{k}})^{(\ell)} \rbb^{\leq \ell^{-r}}$ with eigenvalues in $\{\alpha^n,\alpha^{n+1},\ldots\}$, the $\bbQ_\ell$-linear span of these eigenspaces is dense in $W_{-n}$ with respect to the Gauss norm. As $\sigma$ has only finitely many eigenvalues on $M_n(\bbQ_\ell)$ and $\alpha$ is not a root of unity, it follows that $\tilde\beta$ is zero on~$W_{-n}$ for $n$ large enough. Finally, as by \cite{Litt}, Proposition~2.7, we have $\cI^n \subset W_{-n}$, it follows that $\beta$ is zero on~$\cI^n$ for $n$ large enough, which means that $\tau$ is unipotent.
\end{proof}

We can now prove the main result of this section.

\begin{theorem}
\label{thm:LittUniform}
Let $X/k$ be as in~\emph{\ref{ssec:LittSetup}}. Then there exists an integer~$\ell_X$ such that for all prime numbers $\ell \geq \ell_X$ and all arithmetic representations $\tau \colon \pi_1(X_{\bar{k}}) \to \GL_n(\bbZ_\ell)$ we have
\begin{equation}
\text{$\tau$ is trivial modulo $\ell$} \quad\implies\quad \text{$\tau$ is unipotent.}
\end{equation}
In particular, there exists a positive integer~$N(X)$ such that $N(X,\ell) \leq N(X)$ for all~$\ell$.
\end{theorem}

\begin{proof}
By Corollary~\ref{cor:indexbdd} and Lemma~\ref{lem:nusurj} there exist integers $L\geq 2$ and~$M$ such that for all prime numbers $\ell > L$ the image of $\Image(\rho_{\ell,X}) \cap \sQ_\ell(\bbZ_\ell)$ under~$\nu$ has index less than~$M$ in~$\bbZ_\ell^\times$. This means that for every $\ell > L$ we can find an element $g_\ell \in \Image(\rho_{\ell,X}) \cap \sQ_\ell(\bbZ_\ell)$ such that $\alpha_\ell = \nu(g_\ell)$ is of the form
\[
\alpha_\ell = z_\ell \cdot \exp(y_\ell)\qquad \text{(with $z_\ell \in \bbZ_\ell^\times$ a root of unity and $y_\ell \in \ell\bbZ_\ell$)}
\]
such that the order of $z_\ell$ is at least $(\ell-1)/M$ and $v_\ell(y) < 1+  \log_\ell(M)$. Then $\lim_{\ell\to\infty} C(\alpha_\ell,\ell) = 0$ and by Proposition~\ref{prop:2CBound} this gives the result.
\end{proof}

{\small

\bigskip

} 

\noindent
\texttt{anna.cadoret@imj-prg.fr}

\noindent
IMJ-PRG -- Sorbonne Universit\'e, Paris, France
\medskip

\noindent
\texttt{b.moonen@science.ru.nl}

\noindent
Radboud University, IMAPP, Nijmegen, The Netherlands

\end{document}